\documentclass[12pt]{amsart}
\usepackage{amsfonts}
\usepackage{amsmath}
\usepackage{graphicx}
\usepackage{amssymb}
\newtheorem{theorem}{Theorem}[section]

\newtheorem{corollary}[theorem]{Corollary}

\newtheorem{lemma}[theorem]{Lemma}
\newtheorem{proposition}[theorem]{Proposition}

\theoremstyle{definition}

\theoremstyle{remark}
\newtheorem*{remark}{Remark}

\numberwithin{equation}{section}

\begin{document}

\title[Free boundary minimal immersions]{Minimal immersions of compact bordered Riemann surfaces with free boundary}
\author{Jingyi Chen}

\date{\today}

\thanks{
This work is partially supported by NSERC}

\address{Department of Mathematics, The University of British
Columbia, Vancouver, B.C. V6T 1Z2, Canada}

\email{jychen@math.ubc.ca}
\author{Ailana Fraser}
\email{afraser@math.ubc.ca}
\author{Chao Pang}
\email{ottokk@math.ubc.ca}

\begin{abstract}
Let $N$ be a complete, homogeneously regular Riemannian manifold of dim$N \geq 3$ and let $M$ be a compact submanifold of $N$. Let $\Sigma$ be a compact orientable surface with boundary. We show that for any continuous $f: \left( \Sigma, \partial \Sigma \right) \rightarrow \left( N, M \right)$ for which the induced homomorphism $f_{*}$ on certain fundamental groups is injective, there exists a branched minimal immersion of $\Sigma$ solving the free boundary problem $\left( \Sigma, \partial \Sigma \right) \rightarrow \left( N, M \right)$, and minimizing area among all maps which induce the same action on the fundamental groups as $f$. Furthermore, under certain nonnegativity assumptions on the curvature of a $3$-manifold $N$ and convexity assumptions on the boundary $M=\partial N$, we derive bounds on the genus, number of boundary components and area of any compact two-sided minimal surface solving the free boundary problem with low index. 
\end{abstract}

\maketitle

\section{introduction}

Let $M$ be a closed submanifold of a Riemannian manifold $N$. A branched immersion $u: (\Sigma, \partial \Sigma) \rightarrow (N,M)$ of a surface $\Sigma$ with nonempty boundary $\partial \Sigma$ is a {\em minimal surface with free boundary in $M$} if $u(\Sigma)$ has zero mean curvature and $u(\Sigma)$ is orthogonal to $M$ along $u(\partial \Sigma) \subseteq M$. In this article, we study the free boundary problem for minimal immersions of compact bordered Riemann surfaces. The purpose is twofold. In the first part, we prove a general existence theorem for compact bordered Riemann surfaces of any topological type in complete Riemannian manifolds, assuming certain incompressibility conditions. In the second part, we investigate controlling the topology of free boundary minimal surfaces of low index in $3$-manifolds, under certain nonnegativity assumptions on the curvature and convexity assumptions on the boundary of the 3-manifold. 
Our existence result is:

\begin{theorem} \label{thm:existence}
Let $N$ be a complete, homogeneously regular Riemannian manifold of $dim N \geq 3$ and let $M$ be a compact submanifold of $N$. Then,

\begin{enumerate}

\item[$(i)$]
if $\Sigma$ is a compact, connected orientable surface of genus $g$ and with $k \geq 1$ boundary components that is not a disk, and $f: \Sigma \rightarrow N$ is a continuous map with $f\left(\partial \Sigma \right) \subset M$ such that 
\[
      f_{*}: \pi_{1}(\Sigma) \times \pi_{1}(\Sigma, \partial \Sigma) 
      \rightarrow \pi_{1}(N) \times \pi_{1}(N,M)
\]
is injective, then there exists a branched minimal immersion $\left( \Sigma, \partial \Sigma \right) \rightarrow \left( N, M \right)$ solving the free boundary problem, and minimizing area among all maps \\ $\left( \Sigma, \partial \Sigma \right) \rightarrow \left( N, M \right)$ that induce the same action as $f$ on the fundamental groups;

\item[$(ii)$] 
there exists a set of free homotopy classes $\{ \Gamma_j \}$ of closed curves in $M$ such that the elements $\{ \gamma \in \Gamma_j \}$ form a generating set for $\ker i_*$ acted on by $\pi_1(M)$, where
\[
      i_{*}: \pi_{1}(M) \rightarrow \pi_{1}(N)
\]
is the homomorphism induced by the inclusion, and each $\Gamma_j$ can be represented by the boundary of an area minimizing disk that solves the free boundary problem.
\end{enumerate}
\end{theorem}

The disk case, part ({\em ii}) of the above theorem, was already proved by Ye \cite{Y} by a different method. Existence results for disk-type solutions in various settings have been studied by Meeks and Yau \cite{MY}, Jost \cite{J1, J2, J5}, Struwe \cite{St}, Kuwert \cite{KW}, Fraser \cite{F}, among others. Embedded free boundary solutions of prescribed topological type in 3-manifolds with mean convex boundary  were produced by Jost \cite{J2} assuming a Douglas type condition. Recently M. Li \cite{L} proved existence of embedded solutions of controlled topological type in 3-manifolds with no convexity assumption on the boundary.

We take the Sacks-Uhlenbeck approach of working with the perturbed energy. The analytic foundation is already established in the interior \cite{SU1} and in the boundary \cite{F} settings.
Following the ideas of Schoen-Yau \cite{SY} and Sacks-Uhlenbeck \cite{SU2} for closed surfaces, for each conformal structure on the bordered surface we produce an energy-minimizing map which induces the same action on the fundamental group as a given continuous map $f: (\Sigma, \partial\Sigma) \rightarrow (N,M)$, and then we minimize over all conformal structures to produce a branched minimal surface. The key point is to understand the limiting of the conformal structures in the boundary setting. The incompressibility assumptions on the fundamental groups prevent degeneration in the limiting of the conformal structures, in the form of pinching of an interior or boundary circle or of a curve connecting two boundary circles.

When the ambient manifold has positive curvature, there are some strong restrictions on the topology of stable or index one minimal surfaces. 
In the second part of the paper, we investigate the relationship between the topology of free boundary minimal surfaces and the geometry of the ambient manifold, such as nonnegativity of the curvature and convexity of the boundary, by means of the second variation formula. One of the reasons we are interested in understanding the topology of minimal surfaces in a Riemannian manifold is because it provides information about the ambient manifold (e.g. see \cite{Yau}). 

For closed minimal surfaces in 3-manifolds of positive Ricci curvature, it is conjectured that there should exist a bound on the genus of any minimal surface in terms of its Morse index, and it is known that any index 1 surface must have genus at most three (\cite{RS}). For minimal surfaces with free boundary in $3$-manifolds with nonnegative Ricci curvature and weakly convex boundary, we obtain a bound on the genus and number of boundary components of any index 1 free boundary minimal surface.   Also, there is a bound on the area of stable or index 1 solutions in terms of the topology of the surface and a positive lower bound on the ambient scalar curvature.

Now we state the theorem.

\begin{theorem} \label{thm2}
Let $N$ be a 3-dimensional Riemannian manifold with smooth boundary $\partial N$. Suppose $\Sigma$ is a compact orientable two-sided surface of genus $g$ and with $k \geq 1$ boundary components, solving the free boundary problem $\left( \Sigma, \partial \Sigma \right) \rightarrow \left(N, \partial N \right)$.

\begin{enumerate}

\item[{\em (i)}] Suppose $Ric(N)\geq 0$ and $\partial N$ is weakly convex. 
If $\Sigma$ has index $1$, then: 

\begin{enumerate}
\item[{\em a)}] \ $g+k \leq 3$ if $g$ is even; 

\item[{\em b)}] \ $g+k \leq 4$ if $g$ is odd.

\end{enumerate}

\item[{\em (ii)}] Suppose the scalar curvature $R(N)\geq 0$ and $\partial N$ is weakly mean convex. \\
If $\Sigma$ is stable, then $\Sigma$ is either a disk or a totally geodesic and flat cylinder.  \\
If $\Sigma$ has index 1, then: 

\begin{enumerate}

\item[{\em a)}] \ $k \leq 5$ if $g$ is even; 

\item[{\em b)}] \ $k \leq 7$ if $g$ is odd.

\end{enumerate}

\item[{\em (iii)}] Suppose $N$ has scalar curvature $R \geq R_0 > 0$ and $\partial N$ is weakly mean convex.

\begin{enumerate}

\item[{\em a)}] if  $\Sigma$ is stable, then it is a disk and Area$(\Sigma) \leq \frac{2\pi}{R_0}$;
\item[{\em b)}] if  $\Sigma$ has index 1, then Area$(\Sigma) \leq \frac{2\pi(7-(-1)^{g}-k)}{R_0}$.

\end{enumerate}

\end{enumerate}

\end{theorem}

M. Li \cite{L2} proved an area bound and rigidity result for free boundary minimal surfaces in strictly convex domains in $\mathbb{R}^3$. These area estimates are free boundary analogs of the area estimates for closed stable and index 1 minimal surfaces in $3$-manifolds of positive scalar curvature of Marques and Neves \cite{MN}.
The rigidity result for stable surfaces in part ({\em ii}) can be viewed as the free boundary analog of results of Schoen and Yau for compact ambient manifolds (\cite{SY} Theorem 5.1), and Fischer-Colbrie and Schoen for complete ambient manifolds (\cite{FCS} Theorem 3).

\section{Existence of minimizing harmonic maps in a conjugacy class}

Throughout this paper the terms ``fundamental group" and ``relative fundamental group" will implicitly refer to some base point $*$. Since the isomorphism class of the fundamental group of a path-connected space is not affected by the choice of base-point, without ambiguity, we will write $\pi_{1}( \ \cdot \ )$ and $\pi_{1}(\cdot \ ,\cdot)$ while omitting the base point.

Any continuous map $f$ gives rise to a homomorphism $f_{*}: \pi_{1}( \ \cdot \ , * ) \rightarrow \pi_{1}( \ \cdot \ , f(*))$, which we will call the induced map of $f$. We shall say two maps $f$ and $g$ induce the same action on the fundamental group, if there exists a path $\lambda$ from $f(*)$ to $g(*)$, \linebreak such that $f_{*}=\lambda^{-1}_{*} \circ g_{*} \circ \lambda_{*}$, or equivalently, if $f_{*}$ and $g_{*}$ represent the same homomorphism after identifying the fundamental groups with different base-points through an isomorphism 
$$
  I_{\lambda}: \pi_{1}( \ \cdot \ , f(*)) \rightarrow \pi_{1}( \ \cdot \ , g(*)), \ \ \sigma \mapsto \lambda \cdot \sigma \cdot   \lambda^{-1}.
$$ 
In this case, we will briefly say $f_{*}$ is conjugate to $g_{*}$ and will write $f_{*} \sim g_{*}$ for short.

Finally we recall that the relative fundamental group $\pi_{n}(X,A,*)$ for a triple $\{*\} \subset A \subset X$, is a group only for $n \geq 2$. When $n=1$, it is the set of homotopy classes of paths from the base point $*$ to a varying point in $A$. 

Let $N$, $M$ and $\Sigma$ be as defined in Theorem 1.1. Given  a continuous map $f: \Sigma \rightarrow N$ with $f(\partial\Sigma) \subseteq M$, denote by $f_{*}$ the induced homomorphism as indicated in each of the following situations: 
$$
\begin{array}{lll}
1) & \Sigma \ \text{is not a disk}, & \ f_{*}: \pi_{1}(\Sigma) \times \pi_{1}(\Sigma, \partial \Sigma) \rightarrow \pi_{1}(N) \times \pi_{1}(N,M)\\
2) & \Sigma \ \text{is a disk D}, & \ f_{*}: \pi_{1}(\partial D) \rightarrow \pi_{1}(M).\\
\end{array}
$$
We will use the terminology ``the conjugacy class of $f_{*}$'' to denote the set of maps for which the induced homomorphisms on the above fundamental groups are conjugate to $f_{*}$.

Suppose a conformal structure on $\Sigma$ is fixed, a Riemannian metric compatible with this conformal structure is given, and this metric defines an area element $d\mu$. Let $N \hookrightarrow \mathbb{R}^{K}$ be a $C^{\infty}$ isometric embedding for sufficiently large $K$. Set 
$$
  W^{1,p}(\Sigma,N)=\{u \in W^{1,p}(\Sigma,\mathbb{R}^K) \ | \ u(x) \in N \ \text{a.e.} \ x \in \Sigma \}.
$$ 
For $\alpha>1$, we define the $\alpha$-energy
$$
  E_{\alpha}(u)= \int_{\Sigma} \left(1+|\nabla u|^{2}\right)^{\alpha}d\mu
$$ 
on the admissible space 
$$
  W_{\alpha}=\{u \in W^{1,2\alpha}(\Sigma, N) \ | \ u(\partial \Sigma) \subseteq M, \  u_{*} \sim f_{*}\}.
$$
Note that by the Sobolev embedding theorem, each $u$ in $W^{1,2\alpha}(\Sigma,N)$ is continuous.

\begin{proposition} \label{existence-alpha}
$E_{\alpha}$ attains the infimum at some $u_{\alpha} \in W_{\alpha}$, $\forall \alpha > 1$.
\end{proposition}

\begin{proof}
Let $I_{\alpha}=\inf\limits_{W_{\alpha}}E_{\alpha}$. Let $\{u^{\alpha}_{k}\}$ be an $E_{\alpha}$-minimizing sequence of maps; that is $E_{\alpha}(u^{\alpha}_{k}) \rightarrow I_{\alpha}$. From the Sobolev embedding 
$$W^{1,2\alpha}(\Sigma, N) \hookrightarrow C^{0,\frac{\alpha-1}{\alpha}}(\Sigma,N)$$ 
the sequence $\{u^{\alpha}_{k}\}$ is equicontinuous, so the Arzel\`a-Ascoli theorem yields a subsequence, which we still denote by $\{u^{\alpha}_{k}\}$, that converges uniformly to a map $u_{\alpha}$ in $C^{0,\beta}(\Sigma,N)$ for any $\beta \in \left[0, \frac{\alpha-1}{\alpha}\right)$, and $u_{\alpha}(\partial \Sigma) \subseteq M$. Furthermore, when $k$ is sufficiently large, $u^{\alpha}_{k}$ is homotopic to $u_{\alpha}$, and hence $(u_{\alpha})_{*} \sim (u^{\alpha}_{k})_{*} \sim f_{*}$. On the other hand, from the weak compactness of the unit ball in $W^{1,2\alpha}(\Sigma,N)$, a subsequence of $\{u^{\alpha}_{k}\}$ converges weakly to some $u_{\alpha}'$ in $W^{1,2\alpha}(\Sigma,N)$. It follows that the two limits from the strong convergence and the weak convergence agree, that is $u_{\alpha}=u_{\alpha}' \in W^{1,2\alpha}(\Sigma,N)$. So $u_{\alpha}$ is in $W_{\alpha}$. Now from the lower semi-continuity of the $\alpha$-energy, we have $E_{\alpha}(u_{\alpha})=I_{\alpha}$.
\end{proof}

Next we consider the convergence of a sequence of critical maps of $E_{\alpha}$ as $\alpha \rightarrow 1$.
Notice that for a sequence of minimizing maps $u_{\alpha}$ of $E_{\alpha}$ we have a uniform energy bound. Let $f_{0}$ be a smooth map in the homotopy class of $f$, which exists since $C^{\infty}(\Sigma, N)$ is dense in $C(\Sigma, N)$. Then $f_{0} \in W_{\alpha}$ for all $\alpha >0$. Since $u_{\alpha}$ minimizes $E_{\alpha}$ on $W_{\alpha}$, we have 
\begin{equation*}
      \int_{\Sigma} \left(1+|\nabla u_{\alpha}|^{2}\right)^{\alpha} \,d\mu
      \leq \int_{\Sigma} \left(1+|\nabla f_{0}|^{2}\right)^{\alpha} \, d\mu.
\end{equation*} 
Then for $\alpha \in (1,2)$, we get that the energy of $u_{\alpha}$ is uniformly bounded as 
\begin{equation*}
       \int_{\Sigma}|\nabla u_{\alpha}|^{2} \,d\mu
       < \int_{\Sigma} (1+|\nabla u_{\alpha}|^{2}) \,d\mu
       \leq \int_{\Sigma} \left(1+|\nabla u_{\alpha}|^{2}\right)^{\alpha} \,d\mu
       \leq \int_{\Sigma} \left(1+|\nabla f_{0}|^{2}\right)^{2} \,d\mu.
\end{equation*}

\begin{lemma}
Let $u_{\alpha}$ be a sequence of critical maps of $E_{\alpha}$ with $E(u_{\alpha}) \leq B$. Then a subsequence $u_{\alpha} \rightarrow u$ strongly in $L^2(\Sigma,\mathbb{R}^K)$, weakly in $W^{1,2}(\Sigma,\mathbb{R}^{K})$, and $E(u) \leq \lim\limits_{\alpha \rightarrow 1} E(u_{\alpha})$.
\end{lemma}

Since $N$ is allowed to be noncompact, we impose suitable conditions on $N$.
A complete Riemannian manifold $N$ is  {\em homogeneously regular} if its injectivity radius is bounded from below and its sectional curvature is bounded (see \cite{MSY} p. 623, \cite{M}).
With this condition, and the assumption that the boundary of the surface lies in a compact submanifold of $N$, we can derive the main estimate for critical maps of the $\alpha$-energy at interior points of $\Sigma$ in a similar manner as in the case of closed surfaces in compact manifolds (Proposition 3.2 of \cite{SU1}), as for the free boundary problem in \cite{F} (Proposition 1.7). We then have the following convergence result for critical maps of small energy. 

\begin{lemma}[\cite{SU1}, \cite{F}] \label{epsilonconv}
Let $u_{\alpha}: \Sigma \cap D_{r} \rightarrow N$ with $u_{\alpha}(\partial \Sigma \cap D_{r}) \subset M$ be critical maps of $E_{\alpha}$ such that $u_{\alpha}(\Sigma \cap D_{r})$ meets $M$ orthogonally along $u_{\alpha}(\partial \Sigma \cap D_{r})$ for a sequence $\alpha \rightarrow 1$, that converge in $L^2(\Sigma,\mathbb{R}^{K})$. Then there exists $\epsilon >0$ such that if $E(u_{\alpha})<\epsilon$, then $\{u_{\alpha}\} \rightarrow u$ in $C^{1}(\Sigma \cap \overline{D}_{\frac{r}{2}},N)$ and $u:\Sigma \cap \overline{D}_{\frac{r}{2}} \rightarrow N$ is a smooth harmonic map such that $u(\Sigma \cap D_r)$ meets $M$ orthogonally along $u(\partial \Sigma \cap D_{\frac{r}{2}})$.
\end{lemma}

We can now deduce global convergence of critical maps of the $\alpha$-energy away from a finite number of points.

\begin{theorem}[\cite{SU1}, \cite{F}] \label{globalconvergence}
Let $u_{\alpha}: \Sigma \rightarrow N$ with $u_{\alpha}(\partial \Sigma) \subset M$ be critical maps of $E_{\alpha}$ for a sequence $\alpha \rightarrow 1$, that converge in $L^{2}(\Sigma, \mathbb{R}^{K})$, with $E(u_{\alpha}) < B$. Then there exists a finite set of points $\{z_{1},\cdots,z_{l}\}$ of $\Sigma$ such that $u_{\alpha} \rightarrow u$ in $C^{1}(\Sigma-\{z_{1},\cdots,z_{l}\},N)$ and $u:(\Sigma,\partial \Sigma)\rightarrow (N,M)$ is a smooth harmonic map satisfying the free boundary condition.
Furthermore, when $\Sigma$ is not a disk, if each $u_{\alpha}$ induces the same action on the fundamental group as $f$, then so does $u$.
\end{theorem}

\begin{proof}
The convergence part is Theorem 4.4 in \cite{SU1} and Theorem 1.15 in \cite{F}. We need only verify that the induced map of $u_{\alpha}$ on the fundamental group is preserved in the limiting process, regardless of a finite set of points where bubbling may occur.

Choose, as generators of $\pi_{1}(\Sigma)$, $k+2g$ loops through a base point $*$ in $\Sigma$ such that all of these curves $\{\gamma_j\}$ stay away from the points $\{z_{1},\cdots,z_{l}\}$ where the $C^1$ convergence fails. Since $u_{\alpha} \rightarrow u$ in $C^{1}(\Sigma-\{z_{1},\cdots,z_{l}\},N)$, $u_{\alpha}(\cup_j \gamma_{j})$ is homotopic to $u(\cup_j \gamma_{j})$ for $\alpha$ sufficiently close to 1. It follows that there exists a path connecting $u_{\alpha}(*)$ and $u(*)$ such that $u_{\alpha}(\gamma_{j})$ can be deformed to $u(\gamma_{j})$ along the same path for each $j$. Therefore by definition, $u$ induces the same action as $u_{\alpha}$, and thus $f$ on $\pi_{1}(\Sigma)$. On the other hand, $\pi_{1}(\Sigma, \partial \Sigma)$ is the set of free homotopy classes of paths from a fixed point $* \in \partial \Sigma$ to a varying point on $\partial \Sigma$, and for each class a representative can be chosen away from the points $\{z_{1},\cdots,z_{l}\}$. Therefore $u$ induces the same action as $u_{\alpha}$ for $\alpha$ sufficiently close to 1, and thus $f$ on $\pi_{1}(\Sigma, \partial \Sigma)$.
\end{proof}

\begin{remark}
When $\Sigma$ is a disk $D$, to show the action on $\pi_1$ is preserved in the limit, the argument above cannot be applied as a blowup point may be on a generator of $\pi_1(\partial D)$, and we need a different argument. See Section \ref{disk}.
\end{remark}

The convergence may fail at a finite number of interior or boundary points, where bubbling occurs. Thus the homotopy class of $\{u_{\alpha}\}$ can be altered in the limiting process. However, under stronger conditions on the topology of $N$ and $M$, we obtain the existence of a harmonic map in any homotopy class. This can be interpreted as the free boundary analog of Theorem 5.1 in  \cite{SU1}.

\begin{theorem} \label{homotopyclass}
Let $N$, $M$ and $\Sigma$ be as in Theorem 1.1. If in addition $\pi_{2}(N)=0$ and $\pi_{2}(N,M)=0$, then there exists a minimizing harmonic map satisfying the free boundary condition in every free homotopy class of maps in $C^{0}\left((\Sigma,\partial \Sigma), (N, M)\right)$.
\end{theorem}

\begin{proof}
Let $u_{\alpha}$ be a minimizing map of $E_{\alpha}$ in a fixed homotopy class of $f$ for a sequence $\alpha \rightarrow 1$. Then $\{u_{\alpha}\}$ has uniformly bounded energy by the note after Proposition \ref{existence-alpha}. By Theorem \ref{globalconvergence}, there exists a subsequence such that $u_{\alpha} \rightarrow u$ in $C^{1}(\Sigma - \{z_{1}, \cdots, z_{l}\}, N)$ and $u: (\Sigma, \partial \Sigma) \rightarrow (N,M)$ is a smooth harmonic map satisfying the free boundary condition. We claim that under the topological assumptions of the theorem, $u_{\alpha} \rightarrow u$ in $C^{1}(\Sigma, N)$.

At each point $z_i$ where the $C^1$ convergence fails, center a small disk $D_\rho$ in $\Sigma$ about $z_i$ of radius $\rho$, where $\rho$ is small enough so that $z_j \notin \bar{D}_\rho$ for $j \neq i$. First, assume $z_i$ is an interior point, and choose $\rho$ small enough so that $D_\rho \cap \partial \Sigma = \emptyset$. Let $\eta(r)$ be a smooth function that is $1$ for $r \geq 1$ and $0$ for $r \leq \frac{1}{2}$, and as in \cite{SU1} Theorem 5.1, define a modified map $\widehat{u}_\alpha$ by
\begin{equation} \label{eq:tildeu}
  \widehat{u}_{\alpha}(z)
  =\exp_{u(z)}\left(\eta\big(|z|/\rho\big) \exp^{-1}_{u(z)} \left( u_{\alpha}(z) \right) \right),
\end{equation}
where $\exp$ is the exponential map on $N$. Then $\widehat{u}_\alpha$ agrees with $u_\alpha$ outside $D_\rho$ and with $u$ on $D_{\rho/2}$, and $\widehat{u}_\alpha \rightarrow u$ in $C^1(D_\rho,N)$. By assumption, $\pi_2(N)=0$ and so $u_\alpha$ and $\widehat{u}_\alpha$ are homotopic.

If $z_i$ is a boundary point, let $A \subset \partial \Sigma$ be the segments of the intersection of $\partial \Sigma$ with the annulus $\{\frac{\rho}{2} < |z| < \rho\}$. For the map defined in (\ref{eq:tildeu}), $\widehat{u}_{\alpha}(A)$ may not lie in $M$, however we can modify the map so that it does satisfy the boundary condition. Since $u_\alpha \rightarrow u$ in $C^1$ on $\bar{D}_\rho - D_{\rho/2}$ and $u(\partial \Sigma) \subset M$, we may choose a neighborhood $\Omega_\alpha$ of $A$ in $\bar{D}_\rho - D_{\rho/2}$ that lies in a tubular neighborhood of $\partial \Sigma$ and so that the nearest point projection from $\partial \Omega_\alpha \cap \mbox{int}(\Sigma)$ to $\partial \Sigma$ is one-to-one, such that $\widehat{u}_\alpha(\Omega_\alpha)$ lies in a tubular neighborhood of $M$ in $N$,
with $|\Omega_\alpha|\rightarrow 0$ as $\alpha \rightarrow 1$. On $\Omega_\alpha$, we redefine $\widehat{u}_\alpha$ to map each geodesic segment between a point $z$ of $\partial \Omega_\alpha - A$ and its nearest point in $\partial \Sigma$ proportionally to the geodesic segment in $N$ between $\widehat{u}_\alpha(z)$ and its nearest point in $M$. This modified $\widehat{u}_\alpha$ is piecewise smooth since $u$, $u_\alpha$, the exponential map, and the nearest point projection maps are smooth. Moreover, since $u_\alpha \rightarrow u$ in $C^1$ on $\bar{D}_\rho - D_{\rho/2}$, $|\nabla \widehat{u}_\alpha|$ is bounded independent of $\alpha$ on $\Omega_\alpha$, and since $|\Omega_\alpha| \rightarrow 0$ we have $\lim_{\alpha \rightarrow 1} E_\alpha(\widehat{u}_\alpha|_{\Omega_\alpha})=0$. Finally, since $\widehat{u}_{\alpha}=u_{\alpha}$ on the half circle $\{|z|=\rho\}$, the assumption $\pi_{2}(N,M)=0$ implies that  there exists a homotopy between the disk $u_{\alpha}({D_{\rho}}) \cup  \widehat{u}_{\alpha}({D_{\rho}})$ and a disk in $M$, relative to the boundary $u_{\alpha}(D_{\rho} \cap \partial \Sigma) \cup \widehat{u}_{\alpha}(D_{\rho} \cap \partial \Sigma)$. Hence the two disks bound a $3$-dimensional disk in $N$ and there exists a homotopy between $u_{\alpha}$ and $\widehat{u}_{\alpha}$ relative to the half circle $\{|z| = \rho \}$, mapping the boundary $D_{\rho} \cap \partial \Sigma$ into $M$.

Therefore in either case, whether $z_i$ is an interior or boundary point of $\Sigma$, we have defined a map $\widehat{u}_\alpha$ homotopic to $u_\alpha$ such that
\[
      \lim_{\alpha \rightarrow 1} \tilde{E}_\alpha(\widehat{u}_\alpha|_{D_\rho})=E(u|_{D_\rho}),
\]
where $\tilde{E}_\alpha(u)=\int \left( (1+|\nabla u|^2)^\alpha -1 \right) \, d\mu$.
Since $u_\alpha$ is minimizing for $E_\alpha$ in its homotopy class,  
$E_\alpha(u_\alpha|_{D_\rho}) \leq E_\alpha(\widehat{u}_\alpha|_{D_\rho})$. Therefore,
\[
      \limsup_{\alpha \rightarrow 1} E(u_\alpha|_{D_\rho}) 
      \leq \limsup_{\alpha \rightarrow 1} \tilde{E}_\alpha(u_\alpha|_{D_\rho}) 
      \leq \lim_{\alpha \rightarrow 1} \tilde{E}_\alpha(\widehat{u}_\alpha|_{D_\rho})
      =E(u|_{D_\rho})
      \leq \pi \rho^2 \|u\|_{1,\infty}^2.
\]
Choose $\rho$ sufficiently small so that $\pi \rho^2 \|u \|_{1,\infty}^2 < \epsilon/2$, where $\epsilon$ is as in Lemma \ref{epsilonconv}. Then for $\alpha$ sufficiently close to 1, we have $E(u_\alpha|_{D_\rho}) < \epsilon$ and by Lemma \ref{epsilonconv}, $u_\alpha \rightarrow u$ in $C^1$ on $D_\rho$.
Hence $u_{\alpha} \rightarrow u$ in $C^{1}(\Sigma,N)$ and $u$ is in the same free homotopy class as $f$. 
\end{proof}

We will need the following convergence result for harmonic maps with respect to varying conformal structures on $\Sigma$.

\begin{theorem} \label{convergence-conformal}
Let $u_{i}: (\Sigma, \partial \Sigma) \rightarrow (N,M)$ be a harmonic map satisfying the free boundary condition for a conformal structure $c_{i}$ on $\Sigma$, where $\Sigma$ is not of disk type. Suppose $c_{i}$ converges to a conformal structure $c$ in the $C^{\infty}$-topology and $E(u_{i},c_{i})\leq B$. Then there exist a subsequence $\{u_{i}\}$ and a finite set of points $\{z_{1},\cdots,z_{l}\}$ such that $u_{i} \rightarrow u$ in $C^{1}(\Sigma-\{z_{1},\cdots,z_{l}\},N)$, where $u: (\Sigma, \partial \Sigma) \rightarrow (N,M)$ is a smooth harmonic map satisfying the free boundary condition, and $E(u,c) \leq \liminf_{i \rightarrow \infty}E(u_{i},c_{i})$. Furthermore, if each $u_{i}$ induces the same action on the fundamental group as $f$, then so does $u$.
\end{theorem}

\begin{proof}
The convergence for varying conformal structures on a closed surface lying in a bounded set is shown in \cite{SU2} Theorem 2.3, and the argument carries through for bordered surfaces.
The argument that the induced map of $u_{i}$ on the fundamental group is preserved in the limiting process is as in the proof of Theorem \ref{globalconvergence} above.
\end{proof}

\section{Minimal surfaces of non-disk type}

Recall that the Euler characteristic of a surface of genus $g$ with $k$ boundary components is 
$\chi({\Sigma})= 2-2g-k$. On the disk $D$, there is only one conformal structure, and any smooth harmonic map $u: D \rightarrow N$ with $u(\partial D) \subset M$ and meeting $M$ orthogonally along $u(\partial D)$ is conformal, and hence a branched minimal immersion. When $\Sigma$ is not a disk, that is when $\chi(\Sigma) \leq 0$, in order to produce a branched minimal immersion, we must vary the conformal structure on $\Sigma$.
Let $\mathcal M(\Sigma)$ denote the space of conformal structures on $\Sigma$ with the $C^{\infty}$-topology.
Given a conformal structure $c \in \mathcal {M}(\Sigma)$, let $g$ be a Riemannian metric compatible with $c$ and $d\mu$ be the area element induced by $g$. We then consider 
\[
      E: W^{1,2}(\Sigma,N) \times \mathcal{M}(\Sigma) \rightarrow \mathbb{R}
\]
where
$$E(u,c)= \int_{\Sigma}|\nabla_{g} u|_{g}^{2} \ d\mu.$$ 
By virtue of the conformal invariance of the energy functional, this is independent of the choice of the metric $g$ and is well-defined. We note that $E(\cdot,c)$ is lower semi-continuous in $u$, and $E(u,\cdot)$ is continuous in $c$. 

For each conformal structure $c$ we have produced, by Proposition \ref{existence-alpha} and Theorem \ref{globalconvergence}, when $\Sigma$ is not a disk, a smooth harmonic map $u_{c}$ in the admissible space
$$
  W_f=\{u \in W^{1,2}(\Sigma, N) \ | \ u(\partial \Sigma) \subseteq M, \  u_{*} \sim f_{*}\}
$$ 
such that $E(u_{c},c)=\inf\limits_{u\in W_f}E(u,c)$ (note that for $f \in W^{1,2}(\Sigma, N)$, $f_{*}$ can be defined as in section 1, Schoen-Yau \cite{SY}). Now we define a functional on the space of conformal structures $\bar{E}: \mathcal M(\Sigma) \rightarrow \mathbb{R}$ by 
$$
  \bar{E}(c)=\inf\limits_{u \in W_f}E(u,c)=E(u_{c},c).
$$

\begin{lemma} \label{lem3.1}
$\bar{E}$ is continuous on $\mathcal {M}(\Sigma)$.
\end{lemma}

\begin{proof}
Let $c_{i} \rightarrow c$ be a sequence of conformal structures. Suppose $\bar{E}(c_{i})=E(u_{i},c_{i})$ and $\bar{E}(c)=E(u,c)$ for some $u_{i}, u \in W_f$. 
Let $K=\liminf E(u_i,c_i)$, and let $\{ (u_{i_k},c_{i_k})\}$ be a subsequence such that
$E(u_{i_k},c_{i_k}) \rightarrow K$. By Theorem \ref{convergence-conformal} there exists a further subsequence, which we continue to denote by $\{u_{i_k}\}$, and $u_0 \in W_f$ such that
$E(u_0,c) \leq \lim_{k \rightarrow \infty} E(u_{i_k},c_{i_k})$. Then,
\begin{align*}
       E(u,c) &= \lim_i E(u,c_i) \\
                   &\geq \limsup_{i} E(u_i,c_i)   \qquad \mbox{(since $u_i$ is minimizing for $c_i$)} \\
                   &\geq \liminf_i E(u_i,c_i) \\
                   &= \lim_k E(u_{i_k},c_{i_k}) \\
                   &\geq E(u_0,c) \\
                   & \geq E(u,c) \qquad \qquad \qquad  \mbox{(since $u$ is minimizing for $c$)} 
\end{align*}
It follows that $\bar{E}(c)=E(u,c)=\lim\limits_{i} E(u_{i},c_{i})=\lim\limits_{i} \bar{E}(c_{i})$.
\end{proof} 

Suppose $\inf\limits_{\mathcal{M}(\Sigma)} \bar{E}$ is attained at $c\in\mathcal M(\Sigma)$. Let $u$ be a minimizing harmonic map with respect to $c$. Then for any pair $(u',c') \in W_f \times \mathcal M(\Sigma)$, we have 
\begin{equation*}
  E(u,c)=\inf\limits_{W_f}(\cdot,c)=\bar{E}(c) \leq \bar{E}(c')=\inf\limits_{W_f}(\cdot,c') \leq E(u',c').
\end{equation*}

The relationship between such a minimizing pair and the minimal immersion problem is illustrated in the following (c.f. Theorem 1.8 \cite{SU1} and Corollary 1.5  \cite{SU2}).

\begin{theorem} \label{thm3.2}
If $(u,c)$ is a critical point of $E$ on $W_f \times \mathcal M(\Sigma)$, then $u$ is a branched minimal immersion.
\end{theorem}

\begin{proof}
Any variation of the metric arises from a composition of a conformal change in the metric, and a curve in $\mathcal M(\Sigma)$. Hence by the conformal invariance of $E$, the fact that $c$ is a critical point of $E(u,\cdot)$ on $\mathcal M(\Sigma)$ implies that $E$ is critical with respect to any variation of the initial metric induced by $c$. The computation of Sacks and Uhlenbeck (\cite{SU1}, p.6) shows that $u$ is weakly conformal in the interior of $\Sigma$. Then by Gulliver-Osserman-Royden \cite{GOR}, $u$ is a branched minimal immersion.
\end{proof}

\begin{corollary}   \label{cor3.3}
If $(u,c)$ is a minimizer of $E$ on $W_f \times \mathcal {M}(\Sigma)$ with respect to all smooth variations of $c$ preserving the action on the fundamental group, then $u$ minimizes area among all branched immersions having the same action.
\end{corollary}

Let $\mathcal D(\Sigma)$ denote the topological group of diffeomorphisms of $\Sigma$ onto itself with the $C^{\infty}$-topology of uniform convergence on compact sets of all differentials. Let $\mathcal M(\Sigma)$ denote the space of conformal structures on $\Sigma$. There is a natural action
$$
  \mathcal M(\Sigma) \times \mathcal D(\Sigma) \rightarrow \mathcal M(\Sigma)
$$ 
by pulling back metrics. The Riemann moduli space of $\Sigma$ is defined as the quotient 
$$
  \mathcal R(\Sigma)=\mathcal M(\Sigma)/\mathcal D(\Sigma)
$$ 
consisting of equivalent conformal structures with respect to this action.

Let $\{c_{i}\}$ be an $\bar{E}$-minimizing sequence, i.e. $\bar{E}(c_{i}) \rightarrow \inf\bar{E}$. If a subsequence converges to a conformal structure $c$, then $\bar{E}(c)=\inf\bar{E}$ by Lemma \ref{lem3.1}. In fact it suffices to have a weaker condition that $\{c_{i}\}$ converges in the level of the moduli space.
\begin{lemma} \label{lemma:convergenceconformal}
Let  $\{c_{i}\}$ be an $\bar{E}$-minimizing sequence. If there exist diffeomorphisms $\phi_{i} \in \mathcal D(\Sigma)$ and $c \in \mathcal M(\Sigma)$, such that $\phi _{i}^{*}c_{i} \rightarrow c$ in the $C^{\infty}$-topology, then there exists a minimizing conformal structure for $\bar{E}$. 
\end{lemma}

\begin{proof} 
Let $u_{i}$ be a minimizing harmonic map for $\phi_{i}^{*}c_{i}$ which induces the same action as $f\circ\phi_{i}$. Then
 $u_{i} \circ \phi_{i}^{-1} \ \text{induces the same action as} \ f$
and we have 
$$
  \bar{E}(c_{i})\leq E(u_{i} \circ  \phi_{i}^{-1},c_{i})= E(u_{i},\phi_{i}^{*}c_{i})
\leq E(u_{c_{i}} \circ \phi_{i},\phi_{i}^{*}c_{i})=E(u_{c_{i}},c_{i})= \bar{E}(c_{i})
$$ 
where $u_{c_{i}}$ denotes a minimizing harmonic map for $c_{i}$ which induces the same action as $f$. Therefore,
\begin{equation} \label{energy}
  E(u_{i},\phi_{i}^{*}c_{i})=\bar{E}(c_{i}).
\end{equation}
By Theorem \ref{convergence-conformal},  there exists a subsequence $u_{i_k}$ such that 
$u_{i_k} \rightarrow u$ in $C^{1}$ on $\Sigma$ minus a finite set of points, and since $\phi _{i}^{*}c_{i} \rightarrow c$, we have 
\begin{equation} \label{energyineq}
  E(u,c)\leq\lim\inf E(u_{i_k},\phi_{i_k}^{*}c_{i_k}).
\end{equation}
For sufficiently large $k$, $u\circ \phi_{i_k}^{-1}$ induces the same action as $u_{i_k} \circ \phi_{i_k}^{-1}$ on the fundamental group by the proof of Theorem \ref{convergence-conformal}, hence is in the admissible space $W_f$, since $u_{i} \circ \phi_{i}^{-1}$ induces the same action as $f$. Then we have
\begin{eqnarray*}
  \bar{E}((\phi_{i_k}^{-1})^{*}c)
  &\leq& E(u\circ \phi_{i_k}^{-1},(\phi_{i_k}^{-1})^{*}c) \\
  &=&E(u,c) \\
  &\leq &\lim\inf E(u_{i_k},\phi_{i_k}^{*}c_{i_k}) \\
  &=&\lim\inf\bar{E}(c_{i_k}) \\
  &=&\inf\limits_{\mathcal M(\Sigma)}\bar{E} 
\end{eqnarray*} 
where the first equality is by conformal invariance of the energy, the second inequality is by (\ref{energyineq}), the second equality is by (\ref{energy}), and the third equality follows since $\{c_{i_k}\}$ is a minimizing sequence of $\bar{E}$. Therefore 
$$
  \bar{E}((\phi_{i_k}^{-1})^{*}c)=\inf\limits_{\mathcal M(\Sigma)}\bar{E}.
$$ 
This proves the existence of a minimizing conformal structure. In fact, for all large $k$, $(\phi_{i_k}^{-1})^*c$ are $\bar{E}$-minimizers.
\end{proof}

Thus by Theorem \ref{thm3.2}, Corollary \ref{cor3.3} and Lemma \ref{lemma:convergenceconformal}, the minimal area problem is reduced to the convergence problem in the moduli space $\mathcal {R}(\Sigma)$. We will now prove part ($i$) of Theorem \ref{thm:existence}. Throughout we let $\{c_{i}\}$ be an $\bar{E}$-minimizing sequence of conformal structures, and $u_{i}$ will denote a minimizing map for $c_{i}$, with $E(u_i,c_i) <B$.

\

\emph{I. $\Sigma$ is not a cylinder.}

\vspace{2mm}

Assume that $\Sigma$ is a surface with $\chi(\Sigma) <0$. For each conformal structure $c_i$ in the minimizing sequence, consider the doubled conformal surface. Applying the compactification theorem of the moduli space of conformal structures for the closed doubled conformal surfaces (Lemma 4 of Abikoff \cite{A}), there is a subsequence of $\{c_i\}$ (which we continue to denote by $\{c_i\}$) and there are diffeomorphisms $\phi_i$ of $\Sigma$ such that either $\phi_i^*c_i \rightarrow c$ in $C^\infty$ or, $(\Sigma, \phi^*c_i)$ converges to a Riemann surface with nodes $\Sigma_\infty$ corresponding to pinching a set of homotopically nontrivial simple closed curves 
in the doubled surface to nodes $w_m$, $m=1, \ldots n$. In the first case, by Lemma \ref{lemma:convergenceconformal} we are done. In the second case, we have curves $\gamma_m$ in $\Sigma$ which are pinched, each of which is either a closed curve (possibly a boundary component) or a curve between two boundary components of $\Sigma$ (corresponding to a closed curve in the doubled surface that crosses $\partial \Sigma$, and must be reflection invariant across $\partial \Sigma$). We may then argue as in \cite{SU2} Theorem 4.3. We  may choose a nested sequence $\{D_j^m\}$ of closed neighborhoods  of $\gamma_m$ such that $D_j^m$ converges to the node $w_m$ of $\Sigma_\infty$, and for each fixed $j$, the change in the conformal structure on $\Sigma$ as $(\Sigma, \phi_i^*c_i) \rightarrow \Sigma_\infty$ is restricted to the interior of $\cup_{m=1}^n D_j^m$ (\cite{B}). Let $\Sigma_j = \Sigma - \cup_{m=1}^n D_j^m$. By Theorem \ref{convergence-conformal}, there is a subsequence $\{u_i^{(1)}\}$ of $\{u_i\}$ that convergences in $C^1(\Sigma_1 -\{p_1, \ldots, p_{\ell_1}\},N)$ to a smooth harmonic map defined on $\Sigma_1$. Given $\{u_i^{(j-1)}\}$, by Theorem \ref{convergence-conformal}, a subsequence $\{u_i^{(j)}\}$ of $\{u_i^{(j-1)}\}$ converges in $C^1(\Sigma_j - \{p_1, \ldots, p_{\ell_j}\}, N)$ to a smooth harmonic map. Since $E(u_i, c_i) < B$, by Lemma \ref{epsilonconv}, $\ell_j < 4B/\epsilon$. Consider the diagonal sequence $\{u_i^{(i)}\}$ which converges to a harmonic map $u$ in $C^1(\Sigma_\infty'-\{p_1, \ldots, p_\ell\},N)$, with $\ell \leq 4B/\epsilon$, where $\Sigma_\infty'$ is the punctured Riemann surface $\Sigma - \{w_1, \ldots , w_n\}$. Since $E(u_i, c_i)<B$ for all $i$, $E(u)<B$, and by Theorem 1.6 of \cite{SU1} and  Theorem 1.10  of \cite{F}, $u$ can be extended to a smooth harmonic map $u: \tilde{\Sigma}_\infty \rightarrow N$ satisfying the free boundary condition, where $\tilde{\Sigma}_\infty=\Sigma_\infty'\cup \{q_1, \ldots, q_s, (q_{s+1},\,q_{s+1}'), \ldots, (q_n,\,q_n')\}$ is the bordered Riemann surface obtained by adding a point $q_m$ at the punctures of $\Sigma_\infty'$ corresponding to the nodes $w_m \in \Sigma_\infty$ resulting from the pinching of components of $\partial \Sigma$, and adding a pair of points $(q_m,q_m')$ at the two punctures of $\Sigma_\infty'$ (which may be boundary points) corresponding to each node $w_m \in \Sigma_\infty$ resulting from the pinching of a closed curve inside $\Sigma$ or a closed curve in the doubled surface that crosses $\partial \Sigma$. Now let $\gamma$ be a curve homotopic to $\gamma_m$, for any fixed $m$ between $1$ and $n$, chosen to lie in $D_j^m$ for $j$ sufficiently large and so as not to contain any of the points $p_1, \ldots, p_\ell$. Since $\gamma \subset \tilde{\Sigma}_\infty$ is homotopically trivial  (either as a closed curve or a relative curve between boundary components) it follows that $u(\gamma)$ is homotopically trivial. But $\lim_{i \rightarrow \infty} u_i^{(i)}(\gamma)=u(\gamma)$, so $u_i^{(i)}(\gamma)$ is homotopically trivial for $i$ sufficiently large. Since $\gamma$ is homotopically nontrivial in $\Sigma$, this contradicts our assumption that the induced map on the fundamental groups is injective. Therefore the second case cannot occur.

\

\emph{II. $\Sigma$ is a cylinder.}

\vspace{2mm}

A cylinder with a conformal structure can be represented by a parallelogram spanned by the vectors $(1,0)$ and $\xi$ in $\mathbb{R}^2$ with sides corresponding to one of the two generators identified. Two cylinders given by $\xi_1$, $\xi_2$, with the same corresponding sides identified, represent conformally equivalent cylinders if $\xi_2=\tau \xi_1$ for some $\tau \in \text{PSL}(2,\mathbb{Z})$. 

Given our minimizing sequence of conformal structures $\xi_i$, and associated minimizing harmonic maps $u_i$, there exist elements $\tau_i \in \text{PSL}(2,\mathbb{Z})$ such that $\tau_i \xi_i$ lies in the fundamental domain of $\text{PSL}(2,\mathbb{Z})$. If $\text{Im}(\tau_i\xi_i) \leq b < \infty$ for all $i$, then a subsequence of $\{\tau_i\xi_i\}$ converges to $\eta$, and by Lemma \ref{lemma:convergenceconformal} we are done. 

Otherwise, suppose that $\kappa_i=\text{Im}(\tau_i\xi_i) \rightarrow \infty$. Let $\eta_i=\tau_i\xi_i$. Then $v_i=u_i \circ \tau_i^{-1}: (\Sigma,\eta_i) \rightarrow N$ is harmonic and $E(v_i,\eta_i)=E(u_i,\xi_i) \leq B$. We consider the following two cases:

$a)$ If the sides corresponding to $\eta_i$ are identified, then on any cylinder $S^1 \times [0, \kappa)$ we can find a subsequence of $\{v_i\}$ which converges in $C^1(S^1 \times [0,\kappa) - \{z_1, \ldots, z_n\},N)$ to a harmonic map $v: S^1 \times [0,\kappa) \rightarrow N$ with $E(v)<B$. Since $\kappa$ was arbitrary, using a diagonal sequence argument as above, we obtain a harmonic map $v: S^1 \times [0,\infty) \rightarrow N$ with $E(v)<B$. But $S^1 \times [0,\infty)$ is conformally $\bar{D} - \{p\}$ for some $p \in D$, and by Theorem 1.6 in \cite{SU1}, $v$ extends to a smooth harmonic map $v:D \rightarrow N$ providing a homotopy of $v_i(S^1 \times \{q\}) \simeq v(S^1 \times \{q\})$ to a point for suitable $q$ and $i$ sufficiently large.  This implies that the generator $\tau_i^{-1}(S^1 \times \{q\})$ of $(\Sigma, \xi_i)$ is mapped by $u_i$, and hence also by $f$, to a loop homotopic to zero, contradicting the assumption that $f_*:\pi_1(\Sigma) \rightarrow \pi_1(N)$ is injective.

$b)$ If the sides corresponding to $(0,1)$ are identified, then on any strip $[0,1] \times (-\kappa,\kappa)$ we can find a subsequence of $\{v_i\}$ which converges in $C^1([0,1] \times (-\kappa,\kappa) - \{z_1, \ldots, z_n\},N)$ to a harmonic map $v: [0,1] \times (-\kappa,\kappa) \rightarrow N)$ with $E(v)<B$. Since $\kappa$ was arbitrary, we can obtain a harmonic map $v: [0,1] \times \mathbb{R} \rightarrow N$ with $E(v)<B$. But $[0,1] \times \mathbb{R}$ is conformally $\bar{D} - \{p_1,\,p_2\}$ for some $p_1,\,p_2 \in \partial D$, and by Theorem 1.10 in \cite{F}, $v$ extends to a smooth harmonic map $v:\bar{D} \rightarrow N$ providing a homotopy of $v_i([0,1] \times \{q\}) \simeq v([0,1] \times \{q\})$ to a point for suitable $q$ and $i$ sufficiently large. This contradicts the assumption that $f_*:\pi_1(\Sigma,\partial \Sigma) \rightarrow \pi_1(N,M)$ is injective.

We have proved part $(i)$ of Theorem \ref{thm:existence}.

\section{Minimizing disks} \label{disk}

Now we prove part $(ii)$ of Theorem 1.1, that
there exists a set of free homotopy classes $\{ \Gamma_j \}$ of closed curves in $M$ such that the elements $\{ \gamma \in \Gamma_j \}$ form a generating set for $\ker i_*$ acted on by $\pi_1(M)$, where
$
      i_{*}: \pi_{1}(M) \rightarrow \pi_{1}(N)
$
is the homomorphism induced by the inclusion, and each $\Gamma_j$ can be represented by the boundary of an area minimizing disk that solves the free boundary problem $(D,\partial D) \rightarrow (N,M)$.
We will need the following lemma.

\begin{lemma} \label{diskvariation}
Let $u: (D,\partial D) \rightarrow (N,M)$ be a critical map of $E_{\alpha}$ on $W^{1,2\alpha}(D,\partial D;N,M)$. Then $u$ satisfies
\begin{equation} \label{4.1}
  \int\limits_{D} \Big(-\big(1+|\nabla u|^{2}\big)^{\alpha}+\alpha\big(1+|\nabla u|^{2}\big)^{\alpha-1}|\nabla u|^{2}\Big)z \;dxdy=0
\end{equation}  
where $z=x+iy$ is the complex coordinate on the disk $D$.
\end{lemma}

\begin{proof}
Writing $u=u(z,\bar{z})$, we have 
$
  |\nabla u|^{2}=|u_{x}|^{2}+|u_{y}|^{2}=4u_{z} \cdot u_{\bar{z}}
$
and 
$$
  E_{\alpha}(u)=\int_D \big(1+4u_{z} \cdot u_{\bar{z}}\big)^{\alpha}  \; \frac{i}{2} \; dz d\bar{z}.
$$ 
Given a complex number $\beta$, let
$$
  \varphi_{\beta}(z)=\frac{z-\beta}{1-\bar{\beta}z}.
$$ 
Let $\beta(t)$ be a differentiable curve in $\mathbb{C}$ with
$$
  |\beta(t)| < 1, \ \beta(0)=0.
$$
Then $\varphi_{t}=\varphi_{\beta(t)}$ is a family of automorphisms of the unit disk, which map the boundary to the boundary. Now we define a variation of $u$ by
$$
  D \stackrel{\varphi_{t}}{\longrightarrow} D \stackrel{u}{\longrightarrow} N 
$$ 
$$
  z \mapsto w=\varphi_{t}(z) \mapsto u(w,\bar{w})
$$
where
\[
       z=\varphi^{-1}_t(w)=\frac{w+\beta}{1+\bar{\beta}w}.
\]
Then we have
\[
  \frac{\partial w}{\partial z}=\frac{1-|\beta|^{2}}{(1-\bar{\beta}z)^{2}},
\quad 
  \frac{\partial z}{\partial w}=\frac{1-|\beta|^{2}}{(1+\bar{\beta}w)^{2}},
\]
and using $\beta(0)=0$, we compute
\begin{equation} \label{eq:var}
  \left.\frac{\partial}{\partial t}\Big(\frac{\partial w}{\partial z}\Big)\right|_{t=0} = 2\bar{\beta}'(0) z,
  \quad
  \left.\frac{\partial}{\partial t}\Big(\frac{\partial z}{\partial w}\Big)\right|_{t=0} = -2\bar{\beta}'(0) w.
\end{equation} 
We have:
\[
  E_{\alpha}(u \circ \varphi_{t}) 
   = \int_D \left( 1+ \frac{\partial w}{\partial z} \frac{\partial \bar{w}}{\partial \bar{z}} | \nabla u|^2 \right)^\alpha     
             \frac{\partial z}{\partial w}\frac{\partial \bar{z}}{\partial \bar{w}} \;\frac{i}{2} \; dwd\bar{w}. 
\]
We compute
\[
       \frac{\partial w}{\partial z}\frac{\partial \bar{w}}{\partial \bar{z}}
       = \frac{(1+\bar{\beta}w)^2(1+\beta \bar{w})^2}{(1-|\beta|^2)^2}, \;\;
       \mbox{ and } \;\; 
       \frac{\partial}{\partial t} 
       \left.\left(\frac{\partial w}{\partial z}\frac{\partial \bar{w}}{\partial \bar{z}}\right)\right|_{t=0}
       =2\bar{\beta}'(0) w +2\beta'(0)\bar{w}.
\]
Using this and (\ref{eq:var}), we have
\begin{align*} 
     \frac{d}{dt} E_{\alpha}(u\circ \varphi_{t})\Big|_{t=0} 
     &=\int_D  \big(1+|\nabla u|^{2}\big)^{\alpha}
     \big(-2\bar{\beta}'(0) {w} -2{\beta}'(0)\bar{w}\big)\;\frac{i}{2}\;dwd\bar{w} \\
     & \quad + \int_D \alpha |\nabla u|^{2}  \big(1+|\nabla u|^{2}\big)^{\alpha-1}
          \big(2\bar{\beta}'(0) {w} + 2{\beta}'(0)\bar{w}\big)\;\frac{i}{2}\;dwd\bar{w} \\ 
     &= {\beta}'(0) \int_D 2\Big(-\big(1+|\nabla u|^{2}\big)^{\alpha}
     +\alpha\big(1+|\nabla u|^{2}\big)^{\alpha-1}| \nabla u|^{2}\Big) \bar{z}\;dxdy \\
   & \quad +\bar{\beta}'(0)\int_D 2\Big(-\big(1+|\nabla u|^{2}\big)^{\alpha}
      +\alpha\big(1+|\nabla u|^{2}\big)^{\alpha-1}|\nabla u|^{2}   \Big){z}\;dxdy. 
\end{align*}
Since $\beta'(0)$ is arbitrary, we get (\ref{4.1}).
\end{proof}

\begin{corollary} \label{cor:4.2}
Let $u_{\alpha}$ be a sequence of critical maps of $E_{\alpha}$ for a sequence $\alpha \rightarrow 1$. \linebreak If $u_{\alpha} \rightarrow u$ in $C^{1}(\bar{D} - \{p\},N)$ where $p \in \partial D$,  and each $u_{\alpha}$ is nontrivial, then $u$ is not a constant map.
\end{corollary}
\begin{proof}
From (\ref{4.1}), taking the imaginary part, and using the fact that 
$\int_{D} y \;dxdy=0$, we have
$$
  \int\limits_{D} \Big(-\big(1+|\nabla u_\alpha|^{2}\big)^{\alpha}+1+\alpha\big(1+|\nabla u_\alpha|^{2}\big)^{\alpha-1}|\nabla u_\alpha|^{2}\Big)\,y \;dxdy=0.
$$
Note that the integrand is similar to that in the variation formula for the sphere derived in Sacks-Uhlenbeck (Lemma 5.3, page 20, \cite{SU1}). Thus by the same argument, we have for $1 \leq \alpha \leq 2$,
\begin{equation*} 
          \frac{\alpha}{2}(\alpha-1)|\nabla u_\alpha|^{4}
          \leq \frac{-\big(1+|\nabla u_\alpha|^{2}\big)^{\alpha}
          +1+\alpha\big(1+|\nabla u_\alpha|^{2}\big)^{\alpha-1}|\nabla u_\alpha|^{2}} 
          {\big(1+|\nabla u_\alpha|^{2}\big)^{\alpha -2}}
         \leq \big (\alpha-1)|\nabla u_\alpha|^{4}.       
\end{equation*}

Without loss of generality, we can assume $p$ is the point $(0,1) \in \partial D \subset \mathbb{R}^2$. Dividing $D$ into the upper half disk $D^+$ and the lower half disk $D^-$, we have 
\[
     \frac{\alpha}{2}\int\limits_{D^{+}}
     \big(1+|\nabla u_{\alpha}|^{2}\big)^{\alpha -2}|\nabla u_{\alpha}|^{4} \,y\;dxdy 
     \leq - \int\limits_{D^{-}}
     \big(1+|\nabla u_{\alpha}|^{2}\big)^{\alpha -2}|\nabla u_{\alpha}|^{4} \,y\;dxdy.
\] 
Assume $u$ is a constant map. 
Then $u_\alpha$ cannot converge to $u$ in $C^1(D,N)$ (Theorem 1.8 in \cite{F}).
Therefore $p$ is a blowup point: that is (by Lemma 1.16 and p. 957 in \cite{F}),
\[
       b_\alpha=\max_{z\in D} |\nabla u_\alpha(z) | =|\nabla u_\alpha(z_\alpha)| \rightarrow \infty
\]
where  $\lim_{\alpha \rightarrow 1} z_\alpha =p$. Consider the rescaled maps
$\tilde{u}_\alpha(z)=u_\alpha(z_\alpha+b_\alpha^{-1}z)$. As $\alpha \rightarrow 1$, 
the domains of $\tilde{u}_\alpha$ exhaust either the whole plane or a half plane, and (a subsequence) $\{\tilde{u}_{\alpha}\}$ converges in $C^1$ on compact subsets to a nontrivial harmonic map $\tilde{u}$.
If $D_R(0)$ denotes the disk of radius $R$ centered at the origin in the plane or the half plane, then we have
\begin{align*}
     \frac{1}{4}E(\tilde{u}|_{D_R(0)})  
     & \leq  \lim\limits_{\alpha \rightarrow 1} \ 
          \frac{1}{4} \int\limits_{D_R(0)}|\nabla\tilde{u}_{\alpha}|^{2} \,dxdy \\
    & =  \lim\limits_{\alpha \rightarrow 1} \ 
          \frac{1}{4} \int\limits_{D_{R/b_\alpha}(z_\alpha)}|\nabla u_{\alpha}|^{2} \,dxdy \\     
    & =  \lim\limits_{\alpha \rightarrow 1} \ 
         \frac{1}{4} \int\limits_{D_{R/b_\alpha}(z_\alpha)}\Big(|\nabla u_{\alpha}|^{2}
    -\frac{|\nabla u_{\alpha}|^{2}}{1+|\nabla u_{\alpha}|^{2}}\Big) \,dxdy\\    
   & =  \lim\limits_{\alpha \rightarrow 1} \ \frac{1}{4} \int\limits_{D_{R/b_\alpha}(z_\alpha)}
    \frac{|\nabla u_{\alpha}|^{2}}{1+|\nabla u_{\alpha}|^{2}} \cdot |\nabla u_{\alpha}|^{2} \,dxdy\\     
   & \leq  \lim\limits_{\alpha \rightarrow 1} \frac{\alpha}{2} \int\limits_{D_{R/b_\alpha}(z_\alpha)}
      \big(1+|\nabla  u_{\alpha}|^{2}\big)^{\alpha-1}
     \frac{|\nabla u_{\alpha}|^{2}}{1+|\nabla u_{\alpha}|^{2}} \cdot |\nabla u_{\alpha}|^{2} \, y \,dxdy\\ 
    & =  \lim\limits_{\alpha \rightarrow 1}\frac{\alpha}{2} \int\limits_{D_{R/b_\alpha}(z_\alpha)}
    \big(1+|\nabla u_{\alpha}|^{2}\big)^{\alpha-2} |\nabla u_{\alpha}|^{4} \, y \,dxdy \\   
    & \leq  \lim\limits_{\alpha \rightarrow 1} \frac{\alpha}{2}
    \int\limits_{D^{+}}  \big(1+|\nabla u_{\alpha}|^{2}\big)^{\alpha -2}
    |\nabla u_{\alpha}|^{4}  \, y\,dxdy \\   
    & \leq \lim\limits_{\alpha \rightarrow 1} 
      -\int\limits_{D^{-}}  \big(1+|\nabla u_{\alpha}|^{2}\big)^{\alpha -2}
      |\nabla u_{\alpha}|^{4}  \, y\,dxdy \\        
    & =0            
\end{align*}
where the first equality is by the conformal invariance of the energy functional,
the second equality follows since $b_{\alpha} \rightarrow \infty$,  
in the second inequality we have used $y > \frac{1}{2}$ on $D_{R/b_\alpha(z_\alpha)}$, and the last equality follows
since $u$ is a constant map and $u_\alpha \rightarrow u$ in $C^1$ on $D^-$, so $|\nabla u_{\alpha}|^{2} \rightarrow 0$ uniformly. This contradicts the fact that $\tilde{u}$ is nontrivial. Therefore $u$ is not a constant map.
\end{proof}

Now we come back to the specific setting of minimizing disks. Given a basepoint $x_0 \in M$, let
\[
      i_{*}: \pi_{1}(M,x_0) \rightarrow \pi_{1}(N,x_0)
\]
be the homomorphism induced by the inclusion $i$ of $M$ in $N$. Recall that two elements $\gamma$ and $\gamma'$ in $\pi_1(M,x_0)$ determine the same free homotopy class of closed curves in $M$ if and only if they belong to the same orbit $\pi_1(M,x_0)\gamma=\pi_1(M,x_0)\gamma'$ under the usual action of $\pi_1(M,x_0)$ on $\pi_1(M,x_0)$. That is, the set of free homotopy classes of closed curves in $M$ is in one-to-one correspondence with the set of orbits $\pi_1(M,x_0)\gamma \subset \pi_1(M,x_0)$  (for further details see \cite{SU1} p.19).
Given an element $\gamma$ in $\ker i_{*}$, let $\Gamma$ be its associated free homotopy class. Let
$$
  W_{\Gamma}=\{u \in W^{1,\infty}(D,\partial D;N,M): [u(\partial D)] = \Gamma \},
$$  
where we use the notation $[u(\partial D)]$ for the free homotopy class of $u(\partial D)$, and define
$$  
  \mathcal{E}(\Gamma)
  =\min \ \{E(u): u\in W_{\Gamma}\}
  =\lim_{\alpha\rightarrow 1}\min \ \{\tilde{E}_{\alpha}(u): u\in W_{\Gamma}\}
$$
where $\tilde{E}_{\alpha}(u)=\int((1+|\nabla u|^{2})^{\alpha}-1)d\mu$. Note that $ \mathcal{E}(\Gamma) =0$ if and only if $\Gamma$ is trivial, and $\mathcal{E}( \Gamma) > \epsilon_0$ otherwise (\cite{F} Theorem 1.8). 

\begin{lemma} \label{split}
Let $\gamma \in \ker i_{*}$ and let $\Gamma=\pi_1(M,x_0)\gamma$ be its associated free homotopy class. Then either $\Gamma$ can be represented by the boundary of an area minimizing disk solving the free boundary problem, or for any $\delta > 0$ there exist nontrivial free homotopy classes $\Gamma_{1}=\pi_1(M,x_0)\gamma_1, \, \Gamma_{2}=\pi_1(M,x_0)\gamma_1$ where $\gamma_1, \, \gamma_2 \in \ker i_{*}$, such that
$$
  \pi_1(M,x_0)\gamma \subset \pi_1(M,x_0)\gamma_1+\pi_1(M,x_0)\gamma_2, \ \ \ \mathcal{E}(\Gamma_{1}) + \mathcal{E}(\Gamma_{2}) < \mathcal{E}(\Gamma)+\delta.
$$
\end{lemma}
\begin{proof}
Since $\gamma \in \ker i_*$ there exists $f: (D, \partial D) \rightarrow (N,M)$ such that $[f(\partial D)] =\Gamma$. As in the proof of Proposition \ref{existence-alpha}, there exists a minimizing map $u_\alpha$ of $E_\alpha$ on $W_\Gamma$. By Theorem \ref{globalconvergence} there exists a sequence $\alpha \rightarrow 1$ such that $u_{\alpha} \rightarrow u$ in $C^{1}$ on $D$ minus a finite set of points, and $u: (D,\partial D) \rightarrow (N,M)$ is a (possibly trivial) harmonic map satisfying the free boundary condition. If the set of points where the convergence fails is empty, then $u$ is nontrivial, and hence is an area minimizing disk solving the free boundary problem with $[u(\partial D)]=\Gamma$. Otherwise there exists a point $p$ at which $u_{\alpha}$ fails to converge to $u$ in $C^{1}$. Note that $p$ cannot be an interior point. If $p$ is an interior point, then as in the proof of Theorem \ref{homotopyclass} we can define a modified map $\widehat{u}_\alpha$ by (\ref{eq:tildeu}) with $\widehat{u}_\alpha|_{\partial D} =u_\alpha|_{\partial D}$, so $\widehat{u}_\alpha \in W_\Gamma$ and by the same argument as in the the proof of Theorem \ref{homotopyclass} we have ${u}_\alpha \rightarrow u$ in $C^1(D_\rho(p),N)$. Therefore, $p \in \partial D$.

Now observe that given $\rho >0$, we can find a neighborhood $B$ of $p$ in $\bar{D}$, with $|B|<\rho$, that contains no other points where the convergence fails, and  such that there is a conformal diffeomorphism $h: D - \bar{B} \rightarrow B$ leaving $\partial B \cap D$ fixed. The existence of $B$ and $h$ can be seen in the following way. Let $\varphi: D  \rightarrow H$ be a conformal map from the open disk $D$ to the upper half plane $H$ such that $p$ is mapped to the origin and two nearby points $q$ and $q' \in \partial D$ on either side of $p$ are mapped to $1$ and $-1$. We may choose $q$ and $q'$ sufficiently close to $p$ so that $B:=\varphi^{-1}(D^+)$, where $D^+=D\cap \bar{H}$, has area less than $ \rho$ and contains no other points where the convergence fails. Let $S: \bar{D}^+ \rightarrow H-D^+$ be the conformal map $S(z)={1}/{\bar{z}}$, which is the identity map on the half circle. Then we may take $h=\varphi^{-1} \circ S \circ \varphi$.

Using the construction from Theorem \ref{homotopyclass}, we can define a map $\widehat{u}_\alpha$ that  agrees with $u_\alpha$ outside $B$ and with $u$ on neighborhood of $p$ in $B$, and so that 
$\lim_{\alpha \rightarrow 1} \tilde{E}_{\alpha}(\widehat{u}_{\alpha}|_B) = E(u|_B)$.
Now define
$$
  \begin{array}{lll}
  u_{\alpha}^{1} & = & \left \{
  \begin{array}{ll}
  u_{\alpha} & \ \ \ \ \ \text{on} \ D-B\\
  \widehat{u}_{\alpha} & \ \ \ \ \ \text{on} \ B
  \end{array}
  \right.\\ 
  & & \\
  u_{\alpha}^{2} & = & \left \{
  \begin{array}{ll}
  \widehat{u}_{\alpha}\circ h & \ \text{on} \ D-B\\
  u_{\alpha} & \ \text{on} \ B.
  \end{array}
  \right.
  \end{array}
$$

Let $\Gamma_{1}$ and $\Gamma_{2}$ be the free homotopy classes of $u_{\alpha}^{1}(\partial D)$ and $u_{\alpha}^{2}(\partial D)$ respectively. Then $\Gamma \subset \Gamma_{1}+\Gamma_{2}$. By the conformality of $h$, we have
\begin{align*}
  \lim_{\alpha \rightarrow 1} \tilde{E}_{\alpha}(u_{\alpha}^{1}) 
  & =  \lim_{\alpha \rightarrow 1} \tilde{E}_{\alpha}( u_{\alpha}|_{D-B})
            +E(u |_B) \\
  \lim_{\alpha \rightarrow 1} \tilde{E}_{\alpha}(u_{\alpha}^{2}) 
  & =  \lim_{\alpha \rightarrow 1} \tilde{E}_{\alpha}( u_{\alpha} |_B)
            +E(u |_B).
\end{align*}
Choose $\rho$ sufficiently small so that $E(u|_B) \leq \|u\|^2_{1,\infty} |B|  < \|u\|^2_{1,\infty} \rho < \delta/6$. Then if $\alpha$ is sufficiently close to 1, we have
\begin{align*}
  \tilde{E}_{\alpha}(u_{\alpha}^{1}) 
  & \leq  \tilde{E}_{\alpha}( u_{\alpha} |_{D-B})+ \frac{\delta}{3} \\
  \tilde{E}_{\alpha}(u_{\alpha}^{2}) 
  & \leq  \tilde{E}_{\alpha}( u_{\alpha} |_B)+\frac{\delta}{3},
\end{align*}
and
\begin{equation} \label{eq:sum}
  \mathcal{E}(\Gamma_{1}) + \mathcal{E}(\Gamma_{2}) 
  \leq \tilde{E}_{\alpha}(u_{\alpha}^{1})+ \tilde{E}_{\alpha}(u_{\alpha}^{2})       
  \leq \tilde{E}_{\alpha}(u_{\alpha}) + \frac{2\delta}{3} 
  <  \mathcal{E}(\Gamma)+\delta,
\end{equation}
where the last inequality follows since $\{u_\alpha\}$ is a minimizing sequence for $\mathcal{E}(\Gamma)$.
We may assume $\delta < \frac{1}{2}\min\{\epsilon,\epsilon_0\}$. By Lemma \ref{epsilonconv},
$
        \tilde{E}_{\alpha}(u_{\alpha}^{2}) 
        \geq \tilde{E}_{\alpha}(u_{\alpha}|_B) 
        \geq E(u_\alpha|_B)
        \geq \epsilon
$ for $\alpha$ close to 1, 
and so
\[ 
        \mathcal{E}(\Gamma_{1}) \leq \tilde{E}_{\alpha}(u_{\alpha}^{1}) 
        \leq \mathcal{E}(\Gamma)+\delta-\epsilon
        <\mathcal{E}(\Gamma).
\]
Therefore $\Gamma_{1} \neq \Gamma$ and $\Gamma_{2}$ is nontrivial. It remains to show that $\Gamma_1$ is nontrivial. For $\alpha$ sufficiently close to 1, we have
$$
  \tilde{E}_{\alpha}(u_{\alpha}^{1})
  \geq \tilde{E}_{\alpha}(u_{\alpha} |_{D-B}) 
  \geq E(u|_{D-B})-\frac{\delta}{6} 
  > E(u) - \frac{\delta}{3}.
$$
If $u$ is nontrivial then $E(u) \geq \epsilon_0$ (\cite{F} Theorem 1.8), and so  
\[
        \tilde{E}_{\alpha}(u_{\alpha}^{1})>E(u) -\frac{\delta}{3} \geq \epsilon_0-\frac{\delta}{3}>\delta.
\]
If $u$ is trivial, by Corollary \ref{cor:4.2} there must be a second point $p' \neq p$ where the convergence $u_\alpha \rightarrow u$ fails, and $p' \in \partial D-B$. Then by Lemma \ref{epsilonconv}, $\tilde{E}_\alpha(u^1_\alpha) \geq \tilde{E}_\alpha(u_\alpha|_{D-B}) \geq {E}(u_\alpha|_{D-B}) \geq \epsilon$ for $\alpha$ close to 1. In either case, we have $\tilde{E}_{\alpha}(u_{\alpha}^{1}) > \delta$, and then by equation (\ref{eq:sum}),
\[
        \mathcal{E}(\Gamma_{2}) \leq \tilde{E}_{\alpha}(u_{\alpha}^{2}) 
        < \mathcal{E}(\Gamma)+\delta-\tilde{E}(u^1_\alpha)
        < \mathcal{E}(\Gamma)+\delta-\delta
        =\mathcal{E}(\Gamma).
\]
Therefore, $\Gamma_2 \neq \Gamma$ and $\Gamma_{1}$ is nontrivial.
\end{proof}

\begin{theorem}
There exists a set of free homotopy classes $\{ \Gamma_j \}$ of closed curves in $M$ such that the elements $\{ \gamma \in \Gamma_j \}$ form a generating set for $\ker i_*$ acted on by $\pi_1(M,x_0)$, and each $\Gamma_j$ can be represented by the boundary of an area minimizing disk that solves the free boundary problem.
\end{theorem}
\begin{proof}
Let $\{ \Gamma_j \}$ be the free homotopy classes that can be represented by the boundary of an area minimizing disk that solves the free boundary problem. 
Let $P \subset \ker i_*$ be the subgroup generated by $\gamma \in \Gamma_j$. Suppose $P$ is a proper subgroup. 
Let $I =  \inf \mathcal{E} (\Gamma)$  over all free homotopy classes $\Gamma$ with elements $\gamma \in \Gamma$, $\gamma \notin P$.
Then there exists $\Gamma$ such that $\mathcal{E}(\Gamma) <I+ \epsilon_0/2$.

By assumption, $\Gamma$ cannot be represented by the boundary of an area minimizing disk that solves the free boundary problem, and so by Lemma \ref{split} there exist nontrivial $\Gamma_1$ and $\Gamma_2$ with 
$\pi_1(M,x_0)\gamma \subset \pi_1(M,x_0)\gamma_1 + \pi_1(M,x_0)\gamma_2$ and 
$\mathcal{E}(\Gamma_{1}) + \mathcal{E}(\Gamma_{2}) < \mathcal{E}(\Gamma)+\epsilon_0/2$.
Since $\Gamma_1$ and $\Gamma_2$ are nontrivial,  
$\mathcal{E}(\Gamma_j) \geq \epsilon_0$ for $j=1,\,2$. This implies 
$\mathcal{E}(\Gamma_j) < \mathcal{E}(\Gamma) - \epsilon_0 /2 <I$.
Therefore, by assumption the sets $\pi_1(M,x_0)\gamma_j$ are both in $P$, and so
\[
       \pi_1(M,x_0)\gamma \subset \pi_1(M,x_0)\gamma_1 + \pi_1(M,x_0)\gamma_2 \subset P,
\]
a contradiction. Therefore $P=\ker i_*$, and so the elements $\{ \gamma \in \Gamma_j \}$ form a generating set for $\ker i_*$ acted on by $\pi_1(M,x_0)$, such that each $\Gamma_j$ can be represented by the boundary of an area minimizing disk that solves the free boundary problem.
\end{proof}

\section{Topology of minimal surfaces of low index}

Let $N$ be a compact $3$-manifold with smooth boundary $\partial N$. Suppose $\Sigma$ is a compact orientable two-sided minimal surface in $N$ with boundary $\partial \Sigma$ in $\partial N$ solving the free boundary problem $(\Sigma, \partial\Sigma)\rightarrow (N,\partial N)$. We will investigate controlling the genus and the number of boundary components of $\Sigma$ for stable and index 1 minimal surfaces, under certain curvature and boundary assumptions on $N$.

Let $A$ denote the second fundamental form, and $\nu$ denote the unit normal vector field of $\Sigma$ in $N$. Let $\eta$ denote the outward unit conormal of $\Sigma$ and $T$ the unit tangent vector along $\partial \Sigma$. The index form is the quadratic form 
$$
      I(f,f)= \int_{\Sigma} \big(\left|\nabla f\right|^{2}-\left(Ric(\nu) 
         + |A|^{2}\right)f^{2}\big) \, d\mu
         +\int_{\partial \Sigma}\left\langle \nabla_{\nu}\nu,\eta \right\rangle f^{2} \, ds
$$ 
for any normal variational vector field $f\nu$. The index of $\Sigma$ is defined as the number of negative eigenvalues of the associated bilinear form. 
A function $f \in W^{1,2}(\Sigma,\mathbb{R})$ is an eigenfunction of the index form with eigenvalue $\lambda$  if $I(f,g)=\lambda \langle f,g \rangle_{L^2}$ for all $g \in W^{1,2}(\Sigma,\mathbb{R})$: 
$$
    \int_{\Sigma} \left( \nabla f \cdot \nabla g -(Ric(\nu)+|A|^{2})fg \right) \, d\mu 
    + \int_{\partial \Sigma}\left\langle \nabla_{\nu}\nu,\eta\right\rangle fg \, ds
    =\lambda \int_{\Sigma}fg \, d\mu.
$$ 
Integrating by parts gives
$$
     -\int_{\Sigma}\big(\Delta f + (Ric(\nu)+|A|^{2})f + \lambda f \big)g \; d\mu
     + \int_{\partial \Sigma} \big(\frac{\partial f}{\partial \eta}
     +\left\langle \nabla_{\nu}\nu,\eta \right\rangle f\big)g \; ds=0.
$$ 
Equivalently $f$ solves the following Robin-type boundary value problem: 
$$
  \left\{ \begin{array}{ll}
  \Delta f + \big(Ric(\nu)+|A|^{2}\big)f =- \lambda f  & \text{in} \ \Sigma\\ 
 \frac{\partial f}{\partial \eta}+\left\langle \nabla_{\nu}\nu,\eta \right\rangle f=0 & \text{on} \ \partial \Sigma.
  \end{array}\right.
$$
Let $h>0$ be a first  eigenfunction. We want to use a specific function $f$ orthogonal to $h$ and containing information about the topology of $\Sigma$ in the second variation formula (index form) above. Using arguments as in \cite{H}, \cite{LY} page 274 we have the following:

\begin{lemma} \label{hersch}
There exists a conformal map $f: \Sigma \rightarrow S^{2}$ such that $\int_{\Sigma} f h \, d\mu=0$ and $f$ has degree $\leq [\frac{g+3}{2}]$.
\end{lemma}

\begin{proof}
By gluing a disk on each boundary component of $\Sigma$, we may view $\Sigma$ as a domain in a compact surface $\bar{\Sigma}$ of genus $g$. There exists a conformal map from this closed surface to the sphere
\[
      \psi: \bar{\Sigma} \rightarrow S^2
\]
of degree $\leq \left[\frac{g+3}{2}\right]$ (see \cite{FK}). Let $G$ be the group of conformal diffeomorphisms of ${S^2}$. We claim there exists $\varphi \in G$ such that
\[
       \int_{\Sigma} (\varphi \circ \psi) \,h \; d\mu =0.
\]
To see this, recall that the conformal transformation group $G$ contains a subgroup which is homeomorphic to $B^3$. That is, given $a \in B^3$ that is not the origin, let $\theta(a)= a/|a|  \in S^{2}$, and let $\varphi(t)$ be the one parameter family of conformal transformations of the ball $B^3$ that are dilations on the sphere fixing the opposite poles $\theta(a)$ and $-\theta(a)$. In the group $\varphi(t)$ there is a unique conformal automorphism $\varphi_a$ that maps the origin to $a$. Define $H: B^3 \rightarrow B^3$ by
\[
     H(a)= \frac{1}{\int_\Sigma h \, d\mu} \int_{\Sigma} (\varphi_a \circ \psi) h \; d\mu.
\]
As $a$ approaches the boundary $\partial B^3$,
\[
        \varphi_a(S^{2} \setminus \{-a\}) \rightarrow a
\]
and so
\[
        \int_{\Sigma} (\varphi_a \circ \psi) \; d\mu  \rightarrow a \; \int_\Sigma h \; d\mu.
\]
Therefore, $H$ extends continuously to a map $H: \overline{B^3} \rightarrow \overline{B^3}$ which is the identity on $\partial B^3$. By a standard argument in topology, $H$ must be surjective. Therefore there exists $a \in B^3$ such that $H(a)=0$, as claimed. 
\end{proof}

Now we prove Theorem \ref{thm2}.

\begin{proof}[Proof of theorem \ref{thm2}] 
Let $\Sigma$ be a solution to the free boundary problem $(\Sigma,\partial \Sigma) \rightarrow (N, \partial N)$.
Choose a local orthonormal frame $\{e_{1},e_{2},e_{3}\}$ along $\Sigma$ such that $e_{1}=T$ is the positively oriented unit tangent vector and $e_{2}=\eta$ is the outward unit conormal along $\partial \Sigma$, and $e_{3}=\nu$ is the globally defined unit normal to $\Sigma$. 
For $1\leq i<j\leq 3$, let $R_{ijij}$ denote the sectional curvature of $N$ for the section $e_{i}\wedge e_{j}$. Let $R=R_{1212}+R_{1313}+R_{2323}$ be the scalar curvature of $N$, and let
\begin{equation} \label{5.1}
  R_{33}=R_{1313}+R_{2323}
\end{equation}
be the Ricci curvature for $e_{3}=\nu$. Let $K$ denote the Gauss curvature of $\Sigma$. From the Gauss equation and the fact that $\Sigma$ is minimal we have 
\begin{equation} \label{5.2}
  K=R_{1212}-\frac{1}{2}|A|^{2}.
\end{equation}

First we assume that $\Sigma$ has index 1. Let $h \geq 0$ be a first eigenfunction of the index form of $\Sigma$. By Lemma \ref{hersch} there exists a conformal map $f: \Sigma \rightarrow S^{2}$ of degree $\leq [\frac{g+3}{2}]$ 
such that $\int_{\Sigma} f h \, d\mu=0$. Since $\Sigma$ has index 1 and the component functions $f_{i}$ of $f$ are orthogonal to $h$, we have 
$$
  I(f_{i},f_{i})=\int_{\Sigma}\left(\left|\nabla f_{i}\right|^{2}-\left(R_{33} + |A|^{2}\right)f_{i}^{2}\right) \,    d\mu+\int_{\partial \Sigma}\left\langle \nabla_{\nu}\nu,\eta\right\rangle f_{i}^{2} \, ds \geq 0.
$$
Summing over $i$, and using $\sum_{i=1}^3|f_{i}|^{2}=1$, we get 
$$
  \int_{\Sigma}\left( \left|\nabla f\right|^{2}-\left(R_{33} + |A|^{2}\right)\right) \, d\mu+\int_{\partial                    \Sigma}\left\langle \nabla_{\nu}\nu,\eta \right\rangle \, ds \geq 0.
$$ 
Since $f: \bar{\Sigma} \rightarrow S^2$ is conformal,
$$
  \int_{\Sigma}\left|\nabla f\right|^{2} \, d\mu <\int_{\overline{\Sigma}}\left|\nabla f\right|^{2} \,                      d\mu=2\text{Area}(f(\overline{\Sigma})) =2\text{Area}(S^{2})\cdot \text{deg}(f)\leq 8\pi\Big[\frac{g+3}{2}\Big].
$$
Therefore 
\begin{equation} \label{5.3}
  \int_{\Sigma}(R_{33} + |A|^{2}) \, d\mu < 8\pi\Big[\frac{g+3}{2}\Big] + \int_{\partial \Sigma}\left\langle                      \nabla_{\nu}\nu,\eta\right\rangle \, ds.
\end{equation}

Now we prove the three parts of the theorem.

\

\textbf{Part (i)} $Ric(N) \geq 0$ and $\partial N$ is weakly convex.

\vspace{2mm}

From (\ref{5.1}) and (\ref{5.2}) we have
\begin{equation}  \label{5.4}
  R_{33}+2K=R_{11}+R_{22}-|A|^{2}.
\end{equation}
Inserting (\ref{5.4}) into (\ref{5.3}), we get 
\begin{equation} \label{5.5}
  \int_{\Sigma}(R_{11}+R_{22}-2K) \, d\mu  < 8\pi\Big[\frac{g+3}{2}\Big] + \int_{\partial \Sigma}\left\langle                     \nabla_{\nu}\nu,\eta\right\rangle \, ds.
\end{equation}
By the Gauss-Bonnet theorem,
\begin{equation*}
  \int_{\Sigma}K \, d\mu+\int_{\partial \Sigma}k_{g} \,ds=2\pi \chi(\Sigma)=2\pi(2-2g-k),
\end{equation*} 
where $k_{g}$ is the geodesic curvature of $\partial \Sigma$ in $\Sigma$. Recall that $k_{g}=-\left\langle \nabla_{T} T,\eta\right\rangle$, so 
\begin{equation} \label{5.6}
  \int_{\Sigma}K\, d\mu-\int_{\partial \Sigma}\left\langle \nabla_{T} T,\eta \right\rangle \,                                ds=2\pi(2-2g-k).
\end{equation}
Inserting (\ref{5.6}) into (\ref{5.5}), and using the assumption that $Ric(N) \geq 0$, we get
\[
  4\pi(2g+k-2)<8\pi\Big[\frac{g+3}{2}\Big] + \int_{\partial \Sigma}\left\langle \nabla_{\nu}\nu,\eta\right\rangle   \, ds +2\int_{\partial \Sigma}\left\langle \nabla_{T} T,\eta\right\rangle \, ds.
\]
Since $\eta$ is orthogonal to $\partial N$ and $\partial N$ is weakly convex we get 
$$ 
  g+\frac{k}{2}-1<\Big[\frac{g+3}{2}\Big].
$$ 
Since 
$$
  \Big[\frac{g+3}{2}\Big]=\frac{g+3-\frac{1+(-1)^{g}}{2}}{2},
$$ 
it follows 
$$
  g+k+\frac{1+(-1)^{g}}{2}<5.
$$
From this we obtain  $i)$  $g+k \leq 3$ if $g$ is even, $ii)$ $g+k \leq 4$ if $g$ is odd.

\

\textbf{Part (ii)} $R \geq 0$ and $\partial N$ is weakly mean convex. 

\vspace{2mm}

Adding (\ref{5.1}) and (\ref{5.2}), we have 
\begin{equation} \label{5.7}
  R_{33}+K=R-\frac{1}{2}|A|^{2}.
\end{equation}
Inserting (\ref{5.7}) into (\ref{5.3}), we obtain 
\begin{equation}  \label{5.8}
  \int_{\Sigma} (R-K+\frac{1}{2}|A|^{2}) \, d\mu < 8\pi\Big[\frac{g+3}{2}\Big] + \int_{\partial \Sigma}\left\langle               \nabla_{\nu}\nu,\eta\right\rangle \, ds.
\end{equation} 
Then using the nonnegative scalar curvature assumption and (\ref{5.6}), we get 
\begin{equation*}
  -2\pi(2-2g-k)<8\pi\Big[\frac{g+3}{2}\Big]+\int_{\partial \Sigma}\left\langle                                   \nabla_{\nu}\nu,\eta\right\rangle\, ds +\int_{\partial \Sigma}\left\langle \nabla_{T}T,\eta\right\rangle \, ds.
\end{equation*} 
Since $\partial N$ is weakly mean convex we obtain 
$$
  g+\frac{k}{2}-1 < 2\Big[\frac{g+3}{2}\Big].
$$ 
Then 
$$
  g+\frac{k}{2}-1<g+3-\frac{1+(-1)^{g}}{2}.
$$ 
From this we obtain $i)$  $k \leq 5$ if $g$ is even,
 $ii)$ $k \leq 7$ if $g$ is odd.

We now assume that $\Sigma$ is stable; that is, $I(f,f) \geq 0$ for all $f \in W^{1,2}(\Sigma)$. Taking $f$ to be a constant function, we obtain
$$
  \int_{\Sigma}( R_{33} + |A|^{2}) \, d\mu \leq \int_{\partial \Sigma}\left\langle \nabla_{\nu}\nu,\eta\right\rangle \, ds .
$$
Using (\ref{5.7}) we get
$$
  \int_{\Sigma} (R-K+\frac{1}{2}|A|^{2}) \, d\mu \leq \int_{\partial \Sigma}\left\langle \nabla_{\nu}\nu,\eta\right\rangle \, ds. $$
By (\ref{5.6}) we get
\begin{equation} \label{5.9}
  \int_{\Sigma} \big(R+\frac{1}{2}|A|^{2}\big)\, d\mu -2\pi (2-2g-k)\leq \int_{\partial \Sigma}\left\langle       \nabla_{\nu}\nu,\eta\right\rangle \, ds+\int_{\partial \Sigma}\left\langle \nabla_{T}T,\eta\right\rangle \, ds.
  \end{equation}
Then since $R \geq 0$ and $\partial N$ is weakly mean convex, we get
$$
  g+\frac{k}{2}-1 \leq 0.
$$ 
Therefore the only possibilities for $(g,k)$ are $(0,1)$ or $(0,2)$ and $\Sigma$ must be a disk or a cylinder.

If $\Sigma$ is a cylinder, from the above we must have $R=0$ and $|A|^{2}=0$ on $\Sigma$, and $I(1,1)=0$. Therefore $f=1$ satisfies the Jacobi  equation
$$
  \Delta f + (Ric(\nu)+|A|^{2})f =0.
$$
This implies $Ric(\nu)=0$, and then from (\ref{5.7}), $K=0$. So $\Sigma$ is a totally geodesic flat cylinder. 

\

\textbf{Part (iii)} We now derive the area estimates.

\vspace{2mm}

If $\Sigma$ has index 1, from (\ref{5.8}) and (\ref{5.6}) we get 
$$
  R_0 \cdot \text{Area}(\Sigma) \leq 8\pi\Big(\Big[\frac{g+3}{2}\Big]-\frac{1}{2}\big(g+\frac{k}{2}-1\big)\Big),
$$ 
and so 
$$
  \text{Area}(\Sigma) \leq \frac{2\pi(7-(-1)^{g}-k)}{R_0}.
$$

If $\Sigma$ is stable, from (\ref{5.9}) we get 
$$
  R_0\cdot\text{Area}(\Sigma) \leq 4\pi\Big(-\big(g+\frac{k}{2}-1\big)\Big).
$$ 
Then $(g,k)$ has to be $(0,1)$ and $\Sigma$ is a disk. Therefore,
$$
  \text{Area}(\Sigma) \leq \frac{2\pi}{R_0}.
$$
This completes the proof of Theorem \ref{thm2}.
\end{proof}

\begin{remark}
Let $N$ be a compact orientable $3$-manifold with boundary $\partial N \neq \emptyset$. Theorem \ref{thm:existence} and \ref{thm2} imply that if there exists a continuous map from a bordered surface with $g=0$ and $k\geq 3$ or $g\geq 1$ and $k\geq 1$, satisfying the incompressible assumption in Theorem \ref{thm:existence}, then $N$ admits no metric of nonnegative scalar curvature, for which $\partial N$ is weakly mean convex. 

In particular, the assumption is satisfied when $\pi_{1}(N)$ contains a subgroup abstractly isomorphic to the fundamental group of a compact orientable surface with $g\geq 1$ and $k \geq 1$, or $g =0$ and $k \geq 3$, and $\pi_{1}(N,\partial N)=0$. Precisely, the first condition asserts the existence of a compact surface of the same topological type in N, which is incompressible. The second condition can be used to deform each boundary circle of this surface to some loop in $\partial N$ by adding a cylinder to the surface. Thus we obtain a compact orientable surface in $N$ with boundary in $\partial N$ satisfying the incompressible condition. 
\end{remark}

\end{document}